\documentclass[11pt]{amsart}
\usepackage[T1]{fontenc}
\usepackage{mathptmx}

\usepackage{mathtools}
\usepackage{comment}
\usepackage[colorlinks=true,linkcolor=magenta,citecolor=blue]{hyperref}     
\usepackage[capitalise]{cleveref}   
\numberwithin{equation}{section}
\usepackage{mathtools}
\usepackage{amsmath, amssymb,amsthm}
\usepackage{array}
\usepackage{graphicx}
\usepackage{float}
\usepackage{caption,subcaption}
\usepackage{tikz}
\usepackage{tikz-cd}
\usetikzlibrary{matrix,arrows,decorations.pathmorphing}
\usepackage[mathscr]{euscript}
\usepackage{MnSymbol}

\def\A{\mathbb{A}}
\def\Bbb{\mathbb{B}}

\def\R{\mathbb{R}}
\def\C{\mathbb{C}}
\def\Q{\mathbb{Q}}
\def\Z{\mathbb{Z}}
\def\N{\mathbb{N}}
\def\Pbb{\mathbb{P}}

\def\Ocal{\mathcal{O}}

\def\Cfra{\mathfrak{C}}

\def\Cscr{\mathscr{C}}

\def\Gscr{\mathscr{G}}
\def\Cscr{\mathscr{C}}

\def\Xscr{\mathscr{X}}

\newtheorem{thm}{Theorem}[section]
\newtheorem{lem}[thm]{Lemma}
\newtheorem{ex}[thm]{Example}
\newtheorem{rem}[thm]{Remark}

\newtheorem{cor}[thm]{Corollary}
\newtheorem{prop}[thm]{Proposition}

\newtheorem*{hyp*}{Hypothesis}

\newcommand\Div{\mathrm{Div}}

\newcommand\Bs{\mathrm{Bs}}

\newcommand\numberthis{\addtocounter{equation}{1}\tag{\theequation}}
\newcommand\hhat{\hat{h}}
\newcommand{\id}{\mathrm{id}}

\DeclareMathOperator{\ddc}{dd^c}

\DeclareMathOperator{\supp}{supp}

\title{Marked points of families of hyperbolic automorphisms of smooth complex projective varieties}
\author{Yugang ZHANG}
\address{Laboratoire de Math\'ematiques d’Orsay, Université Paris-Saclay, 307 Rue Michel Magat, 91405 Orsay,
France}
\email{\href{yugang.zhang@universite-paris-saclay.fr}{yugang.zhang@universite-paris-saclay.fr}}

\begin{document}
\begin{abstract}
    Let $\pi : X\to \Lambda$ be a flat family of smooth complex projective varieties parameterized by a smooth quasi-projective variety $\Lambda$, and let $f\colon X\to X$ be a family of automorphisms with positive topological entropy. Suppose $\sigma : \Lambda \to X$ is a marked point, i.e., it is a rational section of $\pi$. We propose two methods to measure the stability, normality, or periodicity of the family given by $t \mapsto f_t^n(\sigma(t))$.
    
    First, from an algebraic perspective, we construct geometric canonical height functions that have desirable properties. Second, from an analytic viewpoint, we construct a positive closed $(1,1)$-current with continuous local potential. When $\Lambda$ is a curve, we demonstrate that these two constructions actually coincide, providing a unified approach to understanding the dynamical behavior of the family.
    As an application of the algebraic method, we prove a special case of the Kawaguchi--Silverman conjecture over complex function fields.
\end{abstract}
\maketitle
\section{Introduction}
\subsection{Background and motivation}
Let $f\colon \Lambda\times \Pbb^1_\C \to \Lambda\times \Pbb^1_\C$ be an algebraic family of rational maps of degree $d\geq 2$ parameterized by a smooth complex quasi-projective variety $\Lambda$, defined by $f(t,z)=(t,f_t(z))$. A \emph{marked point} is a morphism $\sigma: \Lambda \to \Pbb_\C^1,$ and it is \emph{stable} if the family $\left\{t\mapsto f_t^n(\sigma(t))\right\}_{n\geq 1}$ is normal on $\Lambda(\C)$. For simplicity, suppose that all the critical points are marked, meaning there exist morphisms $c_i: \Lambda \to \Pbb^1, 1\leq i\leq 2d-2$, such that for all $t\in \Lambda(\C)$, $c_i(t)$ are critical points of $f_t.$ This is always possible by some base change. We say that the family $f$ is \emph{dynamically stable} if all the critically marked points are stable. 

McMullen \cite[Lemma 2.1]{Mcmullen87} proved that if the family $f$ is dynamically stable, then either it is isotrivial --- meaning any two rational maps $f_{t_1}$ and $f_{t_2}$ are conjugate by a Möbius transformation --- or all critical points of $f_t$ are preperiodic for any $t$. Dujardin and Favre \cite[Theorem 2.5]{DFAJM08} generalized this result, showing that for a non-isotrivial family, a critically marked point is stable if and only if it is persistently preperiodic. DeMarco \cite[Theorem 1.1]{DeMarcoANT16} further extended this result to apply to any marked point.

One way to measure the stability of a marked point is through the use of canonical height functions. This theory was originally introduced by Néron and Tate for abelian varieties defined over number fields and later generalized by Call and Silverman to varieties with a polarized endomorphism, defined over a field with a product formula; see \cite{CallSilverman93}. For simplicity, assume that the dimension of $\Lambda$ is one. In this case, the \emph{geometric canonical height} can be defined as the limit
\begin{align*}
    \hhat_f(\sigma)\coloneqq \lim_{n\to+\infty}\frac{1}{d^n}\deg f_t^n(\sigma(t)).
\end{align*}
For a non-isotrivial family $f$, a marked point $\sigma$ is stable if and only if $\hhat_f(\sigma) = 0$ (\cite[Theorem 1.4]{DeMarcoANT16}), which in turn implies that $\sigma$ is preperiodic (\cite{Baker09, Benedetto05}).

To a given family $f$, DeMarco associated a positive closed $(1,1)$-current $\hat{T}_f$, known as the \emph{bifurcation current}, whose support is the non-stable locus of the critically marked points. In particular, the bifurcation current vanishes if and only if the canonical height of all critically marked points vanishes.

The situation in higher dimensional projective space $\Pbb^n$ for $n\geq 2$ is more complicated: a marked point can have canonical height zero but still be non-preperiodic. Nevertheless, Gauthier and Vigny \cite{gauthier2020geometric} established the following. To any (rational) marked point $\sigma\colon \Lambda \to \mathbb{P}_\mathbb{C}^n$, they associated a positive closed $(1,1)$-current and showed that its mass is equal to the canonical height of $\sigma$. They further demonstrated that, outside of an "isotrivial invariant subvariety", $\sigma$ is preperiodic if and only if it has vanishing height (see also \cite{Chatzidakis08} for a model theoretic approach).

The goal of this paper is to establish analogous results for families of automorphisms with positive topological entropy that satisfy the hypothesis (H1) and (H2) (see below). In particular, these results can be applied to any family of automorphisms with ositive topological entropy of smooth complex projective surfaces or projective hyperk\"ahler varieties.

\subsection{Some definitions and notations}\label{intro_sect_def}
\subsubsection*{Dynamical degrees}
Let $f:X\to X$ be a surjective endomorphism of smooth projective variety defined over an algebraically closed field of characteristic zero. Its $p$\emph{-th dynamical degree} $\lambda_p(f)$ is defined as $ \lambda_p(f)\coloneqq \lim_{n\to +\infty}{(f^n)}^*(A^p)\cdot A^{\dim X -p}$, where $A$ is any ample divisor on $X$. This quantity is an interesting birational invariant which reflects various dynamical properties of the map $f.$ For example, when the field is the complex numbers $\C,$ $\lambda_p(f)$ is the spectral radius of the action $f^*$ on $H^{p,p}(f,\R)$. Moreover, Gromov \cite{Gromov03} (for the upper bound part $\leq$) and Yomdin \cite{Yomdin87} proved that $h_\mathrm{top}(f) = \max_{1\leq p\leq \dim X}\log \lambda_p(f)$, where $h_\mathrm{top}(f)$ is the topological entropy of $f.$ Following Cantat in the case of surfaces, when $f$ is biholomorphic and has positive topological entropy, we call it \emph{hyperbolic} or \emph{loxodromic}.

Similar upper bound of topological entropy of holomorphic rational maps were obtained by Dinh and Sibony \cite{DS04ens,DS05}. Over non-archimedean fields, it is a recent result of Favre, Truong and Xie \cite{favre2022topologicalentropyrationalmap}.

\subsubsection*{Isotriviality}
Let $\pi: X\to \Lambda$ be a (surjective) flat family of smooth complex projective varieties. Let $\overline{X}$ and $B$ be smooth compactifications of $X$ and $\Lambda$ such that there exists a surjective morphism $\overline{X} \to B$ extending $\pi.$ By a  slight abuse of notation, we still denote it by $\pi$.

The family $\pi$ is \emph{isotrivial} if there exists an \'etale open subset $V\to \Lambda$ such that $X\times_\Lambda V$ is trivial, i.e., there exists a complex variety $Y$ such that $Y\times_\C V \simeq X\times_\Lambda V$.  Let $f:X\to X$ be a fiber-preserving automorphism, i.e. $\pi \circ f= f$. We say that the family ($\pi,f$) is \emph{isotrivial} (resp. \emph{birationally isotrivial}) if for general parameters $t_1, t_2\in \Lambda(\C)$, $f_{t_1}$ and $f_{t_2}$ are conjugate by an isomorphism (resp. birational map) of the fibers $X_{t_1}$ and $X_{t_2}$. If the pair $(\pi, f)$ is isotrivial, then $\pi$ is also isotrivial (see \cite{Kovacs05}). The converse holds when the automorphism group of a general fiber of $\pi$ is discrete. This condition is satisfied if $\pi$ is a family of projective hyperkähler manifolds (see, e.g., \cite[4.1]{debarre2020hyperkahlermanifolds}). In fact, due to the universal property of the automorphism group (which represents the functor \eqref{authilbert}), there exists a morphism from $V$ to $\text{Aut}(Y)$ whose image is a single point $g$. As a result, $f_t = g$ for all $t \in V(\mathbb{C})$.

\subsubsection*{Two hypotheses}
Let $\pi\colon\, \overline{X} \to B$ be a surjective morphism of smooth complex projective varieties which extends a flat family $X\to \Lambda$ as above. Let $f\colon X\to X$ be a fiber-preserving automorphism.
We suppose that for some (thus for all) $t\in \Lambda(\C)$, the map $f_t$ is hyperbolic with first dynamical degree $\lambda_+ \coloneqq \lambda_1(f_t)$. Similarly denote $\lambda_-\coloneqq \lambda_1(f_t^{-1})$. We introduce the following two hypotheses, denoted by (H1) and (H2), in the same vein as \cite[Definition 2.10]{LesieutreSatriano21}:

\begin{hyp*}~
    \begin{enumerate}
    \item[(H1)] There exists a $\R$-divisor $D^\pm$ on $\overline{X}$ such that $f^*D^\pm|_{\Lambda} \sim_\R \lambda_\pm D^\pm|_{\Lambda}.$ Set $D\coloneqq D^+ + D^-$. For any $t\in \Lambda(\C)$, the restriction $D^\pm_t$ on the fiber $X_t$ is nef, and $D_t$ is big and nef.
    \item[(H2)] The set of points $x\in X$, such that $x$ is saddle periodic in some fiber $X_t$ with respect to the map $f_t$, is Zariski dense in $X.$
\end{enumerate}
\end{hyp*}
Here, a periodic point $x\in X$ above $t$ of exact period $k$ is \emph{saddle} if both eigenvalues of the differential of $f^k_t$ at $x$ are not on the unit circle.

The two hypotheses are satisfied, up to some base change, for families of hyperbolic automorphisms of smooth complex projective surfaces or projective hyperk\"ahler varieties, as discussed in Section~\ref{sectexamples} Examples below.

\subsection{Height functions}\label{introheightfunction}
Let ($\pi,f$) be a family of hyperbolic automorphisms of smooth complex projective varieties satisfying (H1). Fix an ample divisor $M$ on $B$. 

A \emph{marked point} $\sigma: B \to \overline{X}$ is a rational section of $\pi$. Denote by $\Sigma$ the set of marked points. The \emph{height function} $h:\Sigma \to \R$ is defined to be 
\begin{align*}
    h(\sigma)\coloneqq \sigma \cdot D^\pm \cdot \pi^*(M)^{b-1}.
\end{align*}
The \emph{forward/backward geometric canonical height function} $\hhat^\pm_f: \Sigma \to \R$ of $f$ is defined to be 
\begin{align}\label{defHeight+-}
    \hhat^\pm_f(\sigma) \coloneqq \lim_{n\to \pm \infty}\frac{1}{\lambda_{\pm}^n} h(f^n\circ \sigma).
\end{align}
The \emph{geometric canonical height function} of $f$ is 
\begin{align}\label{defHeight}
    \hhat_f\coloneqq \hhat_f^+ + \hhat_f^-.
\end{align}
A marked point $\sigma$ is called \emph{stable} if $\hhat_f(\sigma)=0$.
\begin{thm}[cf.~Propositions~\ref{constructionHeight} and~\ref{stablebound}]\label{thm:main1}
    Under assumption (H1), the canonical height functions $\hhat_f^\pm$ and  $\hhat_f$ are well-defined, non-negative, and uniquely determined by the following two properties.
    \begin{itemize}
        \item There exists a positive constant $C_f>0$ such that for any marked point $\sigma$, we have
            \begin{align*}
                |\hhat_f^\pm(\sigma) - h(\sigma)| < C_f.
            \end{align*}
        \item $\frac{1}{\lambda_\pm} \hhat_f^\pm (f^{\pm 1}\circ \sigma ) = \hhat_f^\pm(\sigma)$.
    \end{itemize}
\end{thm}
The canonical height functions are well-behaved outside some Zariski closed subset.

\begin{thm}\label{thm:main2}
    There exist a positive constant $\varepsilon_f>0$ and an $f$-invariant Zariski closed subset $\Bbb_+\subset X$, i.e., $f(\Bbb_+)=f^{-1}(\Bbb_+)=\Bbb_+$, such that given any marked point $\sigma$ with $\sigma(B)\not\subset \Bbb_+$, we have the weak Northcott property: if $\hhat_f(\sigma)<\varepsilon_f$, then $\hhat_f(\sigma)=0$. In particular, for such $\sigma$, $\hhat_f(\sigma)=0 \iff \hhat_f^+(\sigma)=0 \iff \hhat_f^-(\sigma)=0$.

    If moreover the family ($\pi$, $f$) satisfies $(\mathrm{H}2)$ and is not birationally isotrivial, then the union of the image of the stable marked points is not Zariski dense in $X$.
\end{thm}

The so called weak Northcott property was first obtained by Baker \cite[Theorem 1.6]{Baker09} for families of rational maps of $\Pbb^1$ and generalized by the author to families of polarized endomorphisms \cite{yugang}. It is also known  \cite[Theorem 1.4]{Baker09} that we could not expect to have the exact analogue of the Northcott property as in the number field case, i.e., we could not replace ``there exists a positive constant $\varepsilon_f>0$'' in Theorem \ref{thm:main2} by ``for any positive constant''. We will present an application of this result to the Kawaguchi--Silverman conjecture below.

This theorem is a collection of several propositions from the text. More precisely, the invariance of $\Bbb_+$ is established in Lemma~\ref{corB_+}; the weak Northcott property is addressed in Proposition~\ref{gap}, and its consequence---the equivalence of vanishing of canonical height functions---is demonstrated in Corollary~\ref{corgap}; the last assertion corresponds to Proposition~\ref{propsparsity}.

Thus, the canonical height functions reflect the stability of marked points. In the case of families of surfaces, we can show that stable marked points $\sigma$ are precisely the periodic ones, meaning there exists an integer $m\geq 1$ such that for any $t\in \Lambda(\C)$, we have $f_t^m(\sigma(t))=\sigma(t)$.
\begin{thm}[cf. Corollary~\ref{corK3}]\label{thm:main3}
        Let ($\pi,f$) be a family of hyperbolic automorphisms of smooth complex projective surfaces. Let $\sigma$ be a marked point. If $\sigma(B)\subset \Bbb_+$, then $\hhat_f(\sigma)=0$.
        If the family is not birationally isotrivial, then $\hhat_f(\sigma)=0$ if and only if $\sigma$ is periodic. In the case where the automorphism group of a general fiber coincides with its group of birational maps (e.g., for K3 surfaces), we can relax the condition to require only that the family is non-isotrivial.
\end{thm}

Theorems~\ref{thm:main1}, \ref{thm:main2} and \ref{thm:main3} are the function field analogous to the results of Kawaguchi \cite[Theorem C]{KawaguchiSurface08} and of Lesieutre and Satriano \cite[Theorem 2.27]{LesieutreSatriano21}. Our approach is inspired by \cite{gauthier2020geometric,GVhenon}, where Gauthier and Vigny established similar results for families of polarized endomorphisms and H\'enon maps.

\subsubsection*{Models}
Let $B'$ be a smooth complex projective variety and $q: B' \to B$ a generically finite morphism. A $B'$-model of $\pi :\overline{X} \to B$ is a surjective morphism $\pi': Z \to B'$, where $Z$ is a smooth complex projective variety, such that the generic fibers of $\overline{X}\times_B B' \to B'$ and $\pi'$ are isomorphic. For simplicity, we say that $Z$ is a $B'$-model of $X$ and often omit the reference to $B'$. Let $D$ and $D'$ be divisors on $\overline{X}$ and $Z$,  respectively. we say that $D'$ is a $B'$-\emph{model} of $D$ if the restrictions of $D\times_B\times B'$ and $D'$ on the generic fiber is linearly equivalent. Similarly we can define $B'$-\emph{model} of a line bundle. When we say "up to base change," we mean up to a change of models.

To a marked point $\sigma\colon B\to \overline{X}$, A model $Z$ associates a marked point $\sigma'\colon B' \to Z$ by pullback. Thus The set of marked points of $\pi$ is a subset of the marked points of $\pi'$. Hence if we prove Theorem~\ref{thm:main1} and~\ref{thm:main2} for some model $Z$, then they hold for $X$ as well. In fact, we will work on a ``good'' model constructed in Proposition~\ref{B_+genericfiber}.

The invariant algebraic set $\Bbb_+$ in the theorems mentioned above is, up to some base change, the augmented base locus of $D$, see Proposition~\ref{B_+genericfiber}. It is essential to account for this subset, otherwise, the theorems fail, as shown in Theorem~\ref{thm:main2}. Indeed, a marked point can be non-periodic but still have canonical height zero if its image is contained in $\Bbb_+.$ 

\subsection{Kawaguchi--Silverman Conjecture}
Let $A$ be an ample divisor on $X$. Define 
\begin{align*}
    h_A(\sigma)\coloneqq \sigma(B)\cdot A^b.
\end{align*}
The \emph{arithmetic degree} of $\sigma$ is defined as
\begin{align*}
    \alpha_f(\sigma) \coloneqq \lim_{n\to +\infty}h_A(f^n\circ \sigma)^{1/n}
\end{align*}
provided the limit exists, and it does not depend on the ample divisor $A$. In general, define
\begin{align*}
     \overline{\alpha}_f(\sigma) \coloneqq \limsup_{n\to +\infty}h_A(f^n\circ \sigma)^{1/n} \ \ \mathrm{and}\ \  \underline{\alpha}_f(\sigma) \coloneqq \liminf_{n\to +\infty}h_A(f^n\circ \sigma)^{1/n}.
\end{align*}
In our case, the existence of the limit was already known. However, we will reprove this result as part of the proof of Theorem~\ref{thm:degree}.

We establish a special case of the function field analogue of Kawaguchi-Silverman conjecture \cite{kawaguchiSilverman,SilvermanArithmeticDegree}, which roughly states that the geometric complexity of a map (given by the first dynamical degree) equals the arithmetic complexity of a dense orbit (given by the arithmetic degree).
\begin{thm}[cf. Corollary~\ref{cor:KSconj}]\label{thm:degree}
    Let ($\pi$,$f$) be a family of hyperbolic automorphisms of smooth complex projective varieties satisfying (H1) and (H2). Let $\sigma$ be a section with dense orbit, i.e., the subset $\{f^n\circ\sigma(\Lambda) \mid n\in \N \}$ is Zariski dense in $X$. Suppose the family is not birationally isotrivial (or not isotrivial in the case of a family of K3 surfaces). Then $\alpha_f(\sigma)=\lambda_+$.
\end{thm}
The conjecture was originally proposed for varieties defined over $\overline{\Q}$ and has since been extensively studied. For further details, we refer to Matsuzawa's survey \cite{matsuzawa2023recentadvanceskawaguchisilvermanconjecture}. Thanks to the Northcott property, a possible strategy (which goes back to Silverman \cite{Silverman91}) to prove the conjecture over number fields is to construct an appropriate height function associated with the dynamical system. The function field case, however, is more subtle, as the Northcott property generally does not hold. Nevertheless, as explained in \cite{matsuzawa2023recentadvanceskawaguchisilvermanconjecture}, Theorem~\ref{thm:main2} --- whose number field analogous is a direct consequence of Northcott property --- actually suffices to prove the conjecture. 

Remark that we can not remove the non-birationally-isotrivial (or non-isotrivial) assumption. In fact, Take a smooth complex projective variety $X$ and a hyperbolic automorphism $f$ and consider the constant family $F: X\times \mathbb{P}^1_\C$ parameterized by $t\in\mathbb{P}^1_\C$. Let $x\in X$ be a non-periodic point with Zariksi dense orbit. Then the constant section $\sigma_x(t)=(x,t)$ has dense orbit but its arithmetic degree is equal to 1.

Some other cases of the conjecture over function fields can be found in \cite{dimitrov2015silvermansconjectureadditivepolynomial,gauthier2020geometric,GVhenon,MatsuzawaMathz,XieCrelles}.

\subsection{Green currents}
Another criterion for measuring the stability of a marked point involves the use of Green currents.
A (1,1)-current $T$ on a complex manifold $Y$ of complex dimension $n$ can be locally expressed as a differential form with distribution coefficients, i.e. $T= \sum _{1\leq j,k\leq n} T_{j,k} i \mathrm{d}z_j\wedge \mathrm{d}\overline{z}_k$, where $T_{j,k}$ are distributions. We say that $T$ is \emph{closed} if $\mathrm{d} T=0$ and positive if the distribution $\sum c_j \overline{c}_k T_{j,k}$ is \emph{positive} for any $c=(c_1,\cdots,c_n)\in \C^n.$ We denote the operator $\mathrm{d^c} \coloneqq \frac{1}{2i\pi}(\partial-\overline{\partial})$, so that $\ddc = \frac{i}{\pi}\partial\overline{\partial}$.

\begin{thm}[cf. Proposition~\ref{constructionGreenCurrents} and Proposition~\ref{equalityMassHeight}]\label{thm:main4}
    Let ($\pi,f$) be a family of hyperbolic automorphisms of smooth complex projective varieties. Assume $(\mathrm{H}1)$. Then there exists a positive closed (1,1)-current $\hat{T}^\pm_f$ on $X(\C)$ with continuous local potential and such that $\frac{1}{\lambda_\pm}{(f^{\pm n})}^*\hat{T}^\pm_f=\hat{T}^\pm_f$. It is unique in the sense that if there is another current $R^\pm$ satisfying the above properties and $R^\pm - \hat{T}^\pm_f$ is $\ddc$-exact, then $\hat{T}^\pm_f=R^\pm$. 
    
    Moreover, if $\dim B=1$, then for any marked point $\sigma$, we have
    \begin{align*}
        \int_{\sigma(\Lambda)(\C)}\hat{T}_f^\pm = \hhat_f^\pm (\sigma).
    \end{align*}
\end{thm}
We denote by $\hat{T}^\pm_f$ the \emph{forward/backward green currents of } $f$. Its construction is less straightforward than the case of families of polarized endomorphisms due to the non-ampleness of $D$. We need to construct two currents, a priori different, one is positive closed while the other has continuous local potential. We then demonstrate that these two currents are actually equal. A key element in this process is a rigidity result for Green currents, originally obtained by Cantat for K3 surfaces \cite[Th\'eor\`eme 2.4]{CantatK301} and later generalized to higher dimensions by Dinh and Sibony \cite[Theorem]{DSJAG10}.

Cantat, Gao, Habegger and Xie \cite{CGHXDUke21} proved that in a family of abelian varieties, the betti form vanishes if and only if the canonical height is zero. Gauthier and Vigny \cite{gauthier2020geometric} further showed that these two quantities are actually equal for families of polarized endomorphisms (see also \cite{MR4360003} for the case of families of elliptic curves).

We follow the approach of \cite{gauthier2020geometric}, using a suitable DSH function (\cite{DSActa09}) to compute the mass lose, with the key being Lemma~\ref{DegeEstimate}. Their proof relies on Hilbert's Nullstellensatz, which seems challenging to adapt to our context. We instead rely on the Lelong-Poincar\'e equation. However, this approach requires that our base $B$ be a curve.

\subsection{Examples}\label{sectexamples}
Let $k$ be an algebraically closed field of characteristic zero. Let $S/k$ be a smooth projective variety over $k.$ Let $f\colon S\to S$ be a hyperbolic automorphism with first dynamical degree $\lambda\geq 1$. Define the N\'eron–Severi group $\mathrm{N}^1(S)\coloneqq \Div(S)/\equiv $ to be the quotient of $\Div(S)$ by the numerical equivalence relation~$\equiv$. Denote by $\mathrm{N}_\R^1(S) \coloneqq \mathrm{N}^1(S)\otimes_\Z \R$ and $\mathrm{Nef}(S)$ the nef cone (i.e. the closure of the ample cone) in $\mathrm{NS}(S)_\R$. Since $f^*(\mathrm{Nef}(S))\subset\mathrm{Nef}(S)$, Birkhoff's Perron-Frobenius theorem \cite{Birkhoff67} implies that there exists a nef divisor $D^\pm$ such that $(f^{\pm})^*D^\pm \equiv_\R \lambda D^\pm$. we here work on smooth projective surfaces and projective hyperkh\"alher manifolds and show that they verify the two hypotheses (H1) and (H2), see Sect.~\ref{intro_sect_def}.

\subsubsection*{Surfaces}
We refer to the survey \cite{Cantat14} of Cantat and the references therein for more details on dynamics of automorphisms of compact complex surfaces. See also \cite{CantatK301,DillerFave01, DS04,KawaguchiSurface08,Mcmullen02}.

Let $k=\C$ and $\dim S=2$. We have an intersection pairing $H^{1,1}(S,\R) \times H^{1,1}(S,\R) \to~\R$ defined by ($\alpha,\beta$)=$\int_S \alpha\wedge \beta$. The Hodge index theorem says that it is non-degenerate and of signature (1,$\dim H^{1,1}(S,\R)-1$). The automorphism $f$ induces by pulling back a linear map $f^*$ on $H^{1,1}(S,\R)$ that preserves the intersection pairing. By \cite[Lemma 3.1]{Mcmullen02}, $f^*$ has eigenvalues $\lambda,\lambda^{-1}$ of multiplicities one, and the others are all lie on the unit circle. In particular, any $\R$-divisor $E$ (we refer to the beginning of Sect.~\ref{subsection R divisors} for more details about $\R$-divisors) such that $f^* E \sim_\R \lambda E$ or $(f^{-1})^* E \sim_\R \lambda E$ is nef, where $\sim_\R$ denotes $\R$-linear equivalence. Kawaguchi further proved that \cite[Lemma 3.8]{KawaguchiSurface08} there exist nef divisors $D^+$ and $D^-$ such that $D\coloneqq D^+ + D^-$ is big and nef and $f^*D^+ \sim_\R \lambda D^+$ and $(f^{-1})^*D^- \sim_\R \lambda^{-1} D^-$. In fact, this result remains true for any algebraically closed field of characteristic zero. 

Cantat constructed \cite[Sect 2.3]{CantatK301} two positive closed (1,1)-currents $T^\pm$ which are invariant under $f$, i.e., $\frac{1}{\lambda}(f^{\pm})^*T^\pm = T^\pm$. Moreover $T^\pm$ have continuous local potential so that, following the theory of Bedford-Taylor \cite{BT}, we can take the wedge product $\mu \coloneqq T^+\wedge T^-$. The measure $\mu$ is the unique invariant probability measure (after normalization) with maximal entropy (\cite{BLS4,CantatK301,Dujardin06}), in particular it does not charge pluripolar subsets. Thus, combining with the equidistribution theorem of saddle periodic points of Bedford, Lyubich and Smillie \cite{BLS93distribution,BLS4} and Cantat \cite{CantatK301}, we conclude that saddle periodic points are Zariski dense.

Now let ($\pi: X\to \Lambda,f: X\to X$) be a family of hyperbolic automorphisms of smooth complex projective surfaces. Then hypothesis (H2) is satisfied since saddle periodic points are already Zariski dense in every fiber. Denote by $L$ the function field $\C(B)$ of $B$. It induces a hyperbolic automorphism $f_\eta$ of the generic fiber $X_\eta.$ Denote by $\lambda>1$ its first dynamical degree. By the above, there exist nef divisors $D_\eta^+$ and $D_\eta^-$ on $X_\eta$ such that $D_\eta\coloneqq D_\eta^+ + D_\eta^-$ is big and nef, and $f_\eta^*D_\eta^+ \sim_\R \lambda D_\eta^+$ and $(f^{-1})^*D_\eta^- \sim_\R \lambda^{-1} D_\eta^-$. Up to taking a base change by a finite field extension of $k$, we can suppose that all the above data are defined over $k$. 
Let $D^+$ and $D^-$ be the Zariski closure of $D_\eta^+$ and $D_\eta^-$ on $\overline{X}$, and set $D\coloneqq D^++D^-.$ Then up to reducing $\Lambda$, we have $(f^{\pm})^* D^\pm|_X \sim_\R \lambda D^\pm|_X$. Therefore, it satisfies the hypothesis (H1).

\subsubsection*{Projective hyperkh\"ahler variety}
By a (projective) hyperk\"ahler variety $S$ we mean a simply connected projective manifold such that $H^{2,0}(S,\C)$ is generated by a nowhere-degenerated holomorphic 2-form $\omega_S$, normalized so that $\int_S \left(\omega_S\wedge \overline{\omega_S}\right)^{\dim S /2}=1$. There is a quadratic form $q_S$, called the Beauville-Bogomolov form, defined as follows. If $\alpha\in H^{1,1}(S,\R)$ is a (1,1)-smooth form, then define
$q_S(\alpha)\coloneqq \frac{\dim S}{2}\int_S \alpha^2\wedge  \left(\omega_S\wedge \overline{\omega_S}\right)^{\dim S /2 -1}$. It has signature (1, $\dim H^{1,1}(S,\R) -1$) and preserved by $f^*$; see e.g., \cite[Corollary 23.14]{Calabi-YaoBook} and \cite[Lemma 3]{Seung}. Moreover, since $H^1(S,\Ocal_S)=0$, the group of line bundles $\mathrm{Pic}(S)$ is isomorphic to the N\'eron-Severi group $\mathrm{NS}(S)$. Thus the same argument as in the surface case also applies to hyperk\"ahler varieties and hypothesis (H1) holds. 

Hyperbolic automorphisms of compact hyperkh\"ahler varieties have distinct consecutive dynamical degrees $\lambda_{\dim S- i}(f)= \lambda_i (f) =\lambda_1(f)^i$, $0\leq i\leq \dim S/2$ \cite{oguiso09}. We have two positive closed (1,1)-currents $T^\pm$ which are $f$-invariant as in the case of surfaces, and the measure $\mu\coloneqq (T^+)^{\dim S/2} \wedge (T^-)^{\dim S/2}$ has maximal entropy (=$\dim S/2 \lambda_1(f)$) \cite{DS05JAMS,DSJAG10} and is hyperbolic \cite{Dethelin08}. Thus, the saddle periodic points are Zariski dense \cite[Theorem S.5.1]{Katok_Hasselblatt_1995}, and hypothesis (H2) holds as well.

\subsection*{Acknowledgement}
I would like to thank my PhD advisor Thomas Gauthier for numerous discussions and for his constant support. I would like to thank Nguyen-Bac Dang and Charles Favre for answering my questions on intersection theory. I also would like to thank S\'ebastian Boucksom, Emanuele Macr\`i and Gabriel Vigny for useful discussions.

\section{Relative positive divisors}
\subsection{Definitions and basic properties}\label{subsection R divisors}
Let $k$ be a field of characteristic zero. An $k$-algebraic variety $X$ is a geometrically integral separated scheme of finite type over $k.$ 
Denote by $\Div(X)$ the group of (Cartier) divisors of X. Let $K \coloneqq \Q$ or $\R$. A $K$-divisor is an element of $\Div(X)\otimes_\Z K.$ We denote by $\sim_K$ the $K$-linear equivalence relation: two $K$-divisors are $K$-linear equivalent if their difference is the sum of integral divisors with coefficients in $K.$

A $K$-divisor $D$ is called \emph{effective/ample} if it can be written as $D=\sum d_iD_i$ with $d_i>0$ and $D_i$ effective/ample integral divisors. Moreover if $X$ is projective, Then $D$ is called \emph{big} if a sufficiently large integral multiple of $D$ is $K$-linear equivalent to the sum of a $K$-ample divisor and a $K$-effective divisor. It is called \emph{nef} if for any curve $C$ on $X$, $D\cdot C\geq 0$. By \cite[Theorem 2.2.16]{Lazarsfeld_I}, an equivalent definition of a big and nef divisor $D$ is that $D^{\dim X}>0$. Similarly we can define corresponding notions for $K$-line bundles. Note that in this paper, we use the additive notation for the tensor product of line bundles.

Let $\pi :X \to B$ be a surjective morphism of smooth complex projective varieties. Let $\Lambda\subset B$ be a Zariski open subset and $D$ a $K$-divisor on $X$. We say that $D$ is $\pi$-\emph{nef/big/ample} (resp. $\pi/\Lambda$-\emph{nef/big/ample}) if for a general parameter $t\in B(\C)$ (resp. for any $t\in \Lambda(\C)$), the restriction $D_t$ of $D$ on the fiber $X_t$ is nef/big/ample.

If $Y$ is an irreducible subvariety (or divisor) of $X$, we say it is \emph{horizontal} if $\pi(Y)=B$, and \emph{vertical} otherwise. Note that $\pi$-ample or $\pi$-nef divisors need not be ample or nef. Nevertheless, we still have the positivity of the intersection product with horizontal subvarieties.
\begin{lem}\label{lemrelativeample}
    Let $D$ be a $\pi$-ample divisor. Then there exists an affine open subset $U\subset B$ such that $D|_{\pi^{-1}(U)}$ is ample.
\end{lem}
\begin{proof}
    See \cite[Corollaire 4.6.6]{EGA_II} and \cite[Theorem 1.7.8]{Lazarsfeld_I}.
\end{proof}

\begin{lem}\label{constantintersection}
    Let $\pi: X \to \Lambda$ be a flat family of smooth complex projective varieties of relative dimension $k$. Let $L_1,\cdots,L_k$ be line bundles on $\overline{X}$. Then for $t\in \Lambda(\C)$, we have
    \begin{align*}
        L_1\cdots L_k\cdot X_t = L_{1.\eta}\cdots L_{k,\eta}.
    \end{align*}
\end{lem}
\begin{proof}
    This is a special case of \cite[Proposition 2.1]{favre2022topologicalentropyrationalmap}
\end{proof}
\begin{prop}\label{lemrelativenef}
    Let $D$ be a $\pi$-ample (resp. $\pi$-nef) $\R$-divisor. Then for any horizontal subvariety $Y$ of $X$, $D^{\dim Y}\cdot Y>0$ (resp. $\geq 0$).
\end{prop}
\begin{proof}
    By linearity, We can assume that $D$ is a prime integral divisor. If $D$ is $\pi$-ample, then the positivity follows from Lemma~\ref{lemrelativeample}.
    Now, if $D$ is $\pi$-nef, take any ample divisor $A$ on $X$. From the above, we deduce that $ \left(D^{\dim Y}+\frac{1}{n}A \right) \cdot Y>0.$ The conclusion follows by letting $n\to +\infty.$
\end{proof}

\subsection{The augmented base locus}
Let $k$ be an algebraically closed field of characteristic zero. Let $X$ be a smooth projective $k$-variety. Fix a norm $\lVert \cdot \rVert$ on $\mathrm{N}^1_\R(X)$.

Denote by $\Bs(|D|)$ the base locus of the complete linear system $|D|$ of a divisor $D$. The \emph{stable base locus} of a $\Q$-divisor $D$ is defined to be 
\begin{align*}
    \Bbb(D) \coloneqq \bigcap_{m} \Bs(|mD|)
\end{align*}
where the intersection is taken over all integers $m\geq 1$ such that $mD$ is an integral divisor. 

The \emph{augmented base locus} of a $\R$-divisor $D$ is defined to be
\begin{align*}
    \Bbb_+(D)\coloneqq \Bbb(D-A)
\end{align*}
where $A$ is any ample $\R$-divisor such that $D-A$ is a $\Q$-divisor and the norm $\lVert A \rVert$ of $A$ is sufficiently small. This is well defined by \cite[Proposition 1.5]{ELMNP_Fourier}. Another useful equivalent definition \cite[Definition 1.2]{ELMNP_Fourier} is 
\begin{align}\label{AugmentedBaseLocus2}
    \Bbb_+(D) = \bigcap_{D=A+E} \supp(E),
\end{align}
where the intersection is taken over all decompositions $D=A+E$, with $A$ being an ample $\R$-divisor and $E$ an effective $\R$-divisor.

\begin{ex}\label{exampleB_+}
    Let $S$ be a complex projective surface and $f$ a hyperbolic automorphism, see Example~\ref{sectexamples}. Then $\Bbb_+(D)$ is exactly the union of $f$-periodic curves, by \cite[Proposition 3.1]{KawaguchiSurface08}
\end{ex}

\begin{prop}\label{B_+genericfiber}
    Let $B$ be a smooth complex projective variety and $\eta$ its generic point. Denote by $k=\C(B)$ the function field of $B.$ Let $X_\eta$ be a smooth projective $k-$variety. Let $D_\eta$ be a big $\R$-divisor on $X_\eta.$ Then up to a base change by a finite field extension $K/k$, there exist smooth complex projective variety $X$ with a projection $X\to B$ and a big $\R$-divisor $D$ on $X$ such that 
    \begin{itemize}
        \item the generic fiber of $X$ is isomorphic to $X_\eta;$
        \item the restriction of $D$ to the generic fiber $X_\eta$ is $\R$-linearly equivalent to $D_\eta$;
        \item the restriction of the augmented base locus $\Bbb_+(D)$ of $D$ to the generic fiber $X_\eta$ is exactly that $\Bbb_+(D_\eta)$ of $D_\eta.$ 
    \end{itemize}
\end{prop}
\begin{proof}
    Denote by $k^a$ an algebraic closure of $k$. Pull back $D_\eta$ to $X_{k^a,\eta}\coloneqq X_\eta \times_k k^a$ and denote the divisor by $D^a_\eta.$ Since it is big, there exist a positive rational number $t \in \Q$, an ample and effective divisor $A^a_\eta$ and an effective $\R$-divisor $N^a_\eta$ on $X_{k^a,\eta}$ such that 
    \begin{align*}
        D^a_\eta \sim_\R \frac{1}{t}A^a_\eta+N^a_\eta.    
    \end{align*}
    Still denote by $D^a_\eta$ the divisor $\frac{1}{t}A^a_\eta+N^a_\eta.$
    There exist effective $\R$-divisors $N^a_{j,\eta}, 1\leq j \leq r$ such that $N^a_\eta = \sum_{j=1}^r n_j N^a_{j,\eta}, n_j>0$.
    Take a large integer $l\in \Z$ and small real numbers $\varepsilon_j$ such that $\sum_{j=1}^r \varepsilon_j N^a_{j,\eta} +\frac{1}{l}A^a_\eta$ is ample, $D^a_\eta - \sum_{j=1}^r \varepsilon_j N^a_{j,\eta} -\frac{1}{l}A^a_\eta$ is a $\Q$-divisor and 
    \begin{align*}
         \mathbb{B}_+\left(D^a_\eta\right)=\mathbb{B}\left(D_\eta^a-\sum_{j=1}^r \varepsilon_j N^a_{j,\eta} - \frac{1}{l} A^a_\eta\right).
    \end{align*}
    For a sufficiently divisible and large integer $m$ we have
    \begin{align*}
        \mathbb{B}_+\left(D^a_\eta\right)=\mathrm{Bs}\left(\left|m\left(D^a_\eta-\sum_{j=1}^r \varepsilon_j N^a_{j,\eta} - \frac{1}{l} A^a_\eta\right)\right|\right).
    \end{align*}
    Let $E^a_{1,\eta}, \cdots , E^a_{s,\eta}$ be effective divisors linearly equivalent to $m\left(D^a_\eta-\sum_{j=1}^r \varepsilon_j N^a_{j,\eta} - \frac{1}{l} A^a_\eta\right)$ such that
    \begin{align*}
       \mathbb{B}_+\left(D^a_\eta\right) = \bigcap_{j=1}^s \supp\left(E^a_{j,\eta}\right).
    \end{align*}

    Let $K/k$ be a finite field extension such that all the divisors above are defined over $K$. Thus we have corresponding $\R$-divisors $D_\eta, A_\eta, N_{j,\eta}$ and $E_{j,\eta}$ on $X_{K,\eta}\coloneqq X_\eta \times_k K,$ which still satisfy the above properties. 

    By the ampleness of $A_\eta,$ there exist a smooth model $X$ of $X_{K,\eta}$ and an ample divisor $A$ on $X$ such that its restriction on $X_{K,\eta}$ is $aA_\eta,$ where $a\geq 1$ is a positive integer. Let $N_j$ and $E_j$ be the Zariski closure of $N_{j,\eta}$ and $E_{j,\eta}$ in $X$. Then  
\begin{align*}
    D' \coloneqq \frac{1}{at}A + \sum_{j=1}^r n_j N_j
\end{align*}
is a big $\R$-divisor on $X$. Let $V^+_j , V^-_j, j=1,\cdots,s$ be vertical effective divisors on $X$ such that 
\begin{align*}
    E_j +V^-_j \sim m\left(D'-\sum_{j=1}^r \varepsilon_j N_j - \frac{1}{al} A\right) +V^+_j.
\end{align*}
Denote by
\begin{align*}
    D \coloneqq D'+\frac{1}{m}\sum_{j=1}^s V^+_j.
\end{align*}
Up to taking larger $l$ and smaller $\varepsilon_j$, we can assume that $\sum_{j=1}^r \varepsilon_j \mathrm{N}_j + \frac{1}{al} A$ is ample. Thus
\begin{align*}
    \mathbb{B}_+(D) \subset \mathrm{Bs}\left(\left|m\left(D-\sum_{j=1}^r \varepsilon_j \mathrm{N}_j - \frac{1}{al} A\right)\right|\right) = \mathrm{Bs}\left( \left| m \left(D'-\sum_{j=1}^r \varepsilon_j \mathrm{N}_j - \frac{1}{al} A\right)+\sum_{j=1}^s V^+_j\right| \right).
\end{align*}
The linear equivalence
\begin{align*}
    E_j+V^-_j+\sum_{i\neq j} V^+_i \sim m\left(D'-\sum_{j=1}^r \varepsilon_j \mathrm{N}_j - \frac{1}{al} A\right)+\sum_{j=1}^s V^+_j
\end{align*}
implies that ${\Bbb_+ (D)}_\eta \subset \Bbb_+ (D_\eta).$

The other inclusion is easier and can be deduced from the formula \eqref{AugmentedBaseLocus2} and the fact that the restriction of an ample divisor at the generic fiber remains ample.
\end{proof}

\begin{rem}\label{remmodel}\normalfont 
    In the following, we will work with the model constructed in Proposition~\ref{B_+genericfiber}. Up to shrinking $\Lambda$, we can assume also that $\Bbb_+(D)$ has no vertical irreducible components over $\Lambda$.
\end{rem}

\section{Geometric canonical height functions}
Recall the setting in Sect.~\ref{introheightfunction}. 
\begin{prop}\label{constructionHeight}
    Let $(\pi,f)$ be a family of hyperbolic automorphisms of smooth complex projective varieties satisfying $(\mathrm{H}1)$. The canonical height functions $\hhat_f^\pm$ and  $\hhat_f$ are well-defined and uniquely determined by the following two properties.
    \begin{itemize}
        \item There exists a positive constant $C_f>0$ such that for any marked point $\sigma$, we have
            \begin{align*}
                |\hhat_f^\pm(\sigma) - h(\sigma)| < C_f.
            \end{align*}
        \item $\frac{1}{\lambda_\pm} \hhat_f^\pm (f^{\pm 1}\circ \sigma ) = \hhat_f^\pm(\sigma)$.
    \end{itemize}
    In fact we can choose $C_f\coloneqq \frac{\lambda_\pm}{\lambda_\pm-1}.$
\end{prop}
\begin{proof}
Recall from Sect.~\ref{intro_sect_def} that $\overline{X}$ and $B$ are smooth compactifications of $X$ and $B$. Let $X_0 \coloneqq \overline{X},\, L^\pm_0 \coloneqq \Ocal_{\overline{X}}(D^\pm),\, f_0\coloneqq f$ and $\pi_o \coloneqq \overline{\pi} $. fix some integer $i\geq 0$. Suppose we have constructed $X_j,\, L_j^\pm,\, f_j$ and $\pi_j$ for all $j\leq i$. Denote by $X_{i,\Lambda}\coloneqq \pi_i^{-1}(\Lambda)$. Let $X_{i+1}$ be the desingularization of the morphism
\begin{align*}
    X_{i,\Lambda}\to X_{i,\Lambda}^3 \to X_i^3
\end{align*}
where the first map is ($\id,\, f_i,\, f_i^-$) and the second one is the open immersion. The smooth variety $X_{i+1}$ has three projections to $X_i$ that will be denoted respectively by $\varphi_{i+1},\,g_{i+1}$ and $g_{i+1}^{-1}$. Denote by $\varphi_{i+1,\Lambda}$ (resp. $g_{i+1,\Lambda}$) the restrictions of $\varphi_{i+1}$ (resp. $g_{i+1}$) on $X_{i+1,\Lambda}$. Then define 
\begin{align*}
    f_{i+1}\coloneqq \varphi_{i+1,\Lambda}^{-1} \circ f_i \circ \varphi_{i+1,\Lambda}, \ \ L_{i+1}^\pm\coloneqq \frac{1}{\lambda_\pm} {\left(g_{i+1}^\pm\right)}^* L_i^\pm \ \ \text{and}\ \ \pi_{i+1}\coloneqq \pi_i \circ \varphi_{i+1}.
\end{align*}
We have the following commutative diagram:
\[
\begin{tikzcd}
&  & X_{i+1} \arrow[rrd, "\varphi_{i+1}"'] \arrow[rrrrrd, bend left]   &  &                 &  &  &    \\
X_{i+1} \arrow[rru, "f^\pm_{i+1}", dashed] \arrow[rrd, "\varphi_{i+1}"'] \arrow[rrrr, "g^\pm_{i+1}"] &  &   &  & X_i \arrow[rrr, "\varphi_i"] &  &  & X_{i-1} \\
&  & X_i \arrow[rru, "f^\pm_i", dashed] \arrow[rrrrru, "g^\pm_i", bend right] &  &                &  &  &        
\end{tikzcd}
\]
For all integers $i\geq 1$, we have
\begin{multline}\label{equality1}
    \varphi_{i+1}^*L^\pm_i - L^\pm_{i+1} =\varphi_{i+1}^*(\frac{1}{\lambda_\pm}{g^\pm_{i}}^{*} L^\pm_{i-1}) - \frac{1}{\lambda_\pm}{g^\pm_{i+1}}^{*} L^\pm_{i}\\
    =\frac{1}{\lambda_\pm}{g^\pm_{i+1}}^{*}(\varphi_i^*L_{i-1} - L_i )=\frac{1}{\lambda^{i}_\pm} g^\pm_{i+1} \cdots g^\pm_{2}(\varphi_1^* L^\pm_0- L^\pm_1).
\end{multline}
There exists a vertical divisor $V^\pm$ on $X_1$ such that 
\begin{align}\label{VerDiv}
    \varphi_1^* L^\pm_0- L^\pm_1 = \Ocal_{X_1}(V^\pm).
\end{align}
Since $V^\pm$ is vertical, there exists an effective divisor $N$ on $B$ such that 
\begin{align}\label{VerBound}
    -\pi_1^*(N)<V^\pm<\pi_1^*(N).
\end{align}
Denoting by $\sigma_i \colon B \to X_i$ the marked point induced by $\pi_i$ and $\sigma$, we have 
\begin{align*}
    I^\pm_i\coloneqq  \frac{1}{\lambda_\pm^i}h(f^{\pm i}\circ \sigma) -\frac{1}{\lambda_\pm^{i+1}}h(f^{\pm (i+1)}\circ \sigma)  =\sigma_{i+1}(B)\cdot(\varphi_{i+1}^*L_i-L_{i+1})\cdot \pi_{i+1}^*M^{b-1}. 
\end{align*}
By \eqref{equality1}, we have
\begin{align*}
    |I^\pm_i| \leq \frac{1}{\lambda_\pm^i}\sigma_{i+1}(B) \pi_{i+1}^* N \cdot \pi_{i+1}M^{b-1}=\frac{1}{\lambda_\pm^i} N \cdot M^{b-1}.
\end{align*}
Thus $\hhat_f^\pm = h(\sigma)-\sum_{i\geq 0}I_i$ converge and $|\hhat_f^\pm - h(\sigma)|\leq \frac{\lambda_\pm}{\lambda_\pm - 1}$.
Suppose we have another function $\Tilde{h}^\pm$ verifying (1) and (2). Then for any $n\leq \N$, we have
\begin{align*}
    \left|\Tilde{h}^\pm(\sigma) - \hhat_f^\pm(\sigma) \right| \leq \left|\Tilde{h}^\pm(f^{\pm n}\circ \sigma) - \hhat_f^\pm(f^{\pm n}\circ \sigma) \right|=\left|\frac{1}{\lambda_\pm^n}\Tilde{h}^\pm(\sigma) - \frac{1}{\lambda_\pm^n} \hhat_f^\pm(\sigma) \right| \leq \frac{1}{\lambda_\pm^n} C_f.
\end{align*}
Letting $n\to +\infty$, we have $\Tilde{h}^\pm(\sigma) = \hhat^\pm_f(\sigma)$.
\end{proof}
Recall that a marked point $\sigma: B \to X$ is stable if $\hat{h}_f(\sigma)=0.$ 
\begin{prop}\label{stablebound}
    The height functions $\hhat_f(\sigma)$ and $\hhat_f^\pm(\sigma)$ are non negative. A marked point $\sigma$ is stable if and only if there exists a positive constant $C_s>0$ such that 
    \begin{align*}
        f^n(\sigma(B))\cdot D \cdot \pi^*(M)^{b-1} < C_s
    \end{align*}
    for all $n\in \Z$.
\end{prop}
\begin{proof}
    The non-negativity is an immediate consequence of Lemma \ref{lemrelativenef}. Therefore, if $\sigma$ is stable, $\hhat_f^\pm(\sigma)=0$. In particular, by Proposition \ref{constructionHeight}, $|h(f^{\pm n}\sigma)|= |\hhat_f^\pm (f^{\pm n}(\sigma)) - h(f^{\pm n}\sigma)|<C_f.$ The other implication follows from the construction~\eqref{defHeight+-} of height functions.
\end{proof}

\begin{prop}\label{uniformbound}
    Let $A$ be an ample divisor on $\overline{X}$. 
    \begin{enumerate}
        \item There exists a constant $C(A,D,M)>0$, independent of marked points, such that that we have the following. If there exists a constant $C_A>0$ such that $\sigma(B)\cdot A^b <C_A$, then 
        \begin{align*}
            \sigma(B)\cdot D \cdot \pi^*(M)^{b-1} <C(A,D,M)C_A.
        \end{align*}
        \item There exists a constant $C(M,D)>0$, independent of marked points, such that that we have the following. If there exists a constant $C_M>0$ such that $\sigma(B)\cdot D \cdot \pi^*(M)^{b-1} <C_M$, then 
        \begin{align*}
            \sigma(B)\cdot D^b <C(M,D)C_M.
        \end{align*}
        \item There exists a constant $C(D,A)>0$, independent of marked points, such that that we have the following. If there exists a constant $C_D>0$ such that $\sigma(B)\cdot D^b <C_D$, where~$\sigma(B)\not\subset \Bbb_+(D)$, then 
        \begin{align*}
            \sigma(B)\cdot A^b <C(D,A)C_D.
        \end{align*}
    \end{enumerate}
\end{prop}
We need the following generalization of Siu's inequality.
    \begin{lem}[{\cite{Dang20,Jiang_Li}}]\label{Bac}
        Let $B$ be a projective variety of dimension $b.$ For any nef divisors $\alpha_1, \cdots, \alpha_i$ and big and nef divisor $\beta$, we have
        \begin{align*}
            \alpha_1 \cdots \alpha_i \leq C_b \frac{\alpha_1\cdots\alpha_i\cdot \beta^{b-i}}{\beta^b}\beta^i,
        \end{align*}
        where $C_b$ is a positive constant depending only on the dimension $b$.
    \end{lem}
\begin{proof}[Proof of Proposition \ref{uniformbound}]
    (1)\ Since $A$ is ample, there exists a constant $\alpha>0$ such that $\alpha A-D$ and $\alpha A-\pi^*(M)$ are ample. Thus we can take $C(A,D,M)\coloneqq\alpha^b$.

    (2)\ Write $D = \sum_i d_i D_i$, where $d_i>0$ and $D_i$ is big and nef divisor. Let $r_{i,n}\in \Q$ be rational numbers approaching $d_i.$ as $n\to +\infty$. Take any ample divisor $H$ on $\overline{X}$. Denote by $D_{n}\coloneq \sum_i r_{i,n} D_i + \frac{1}{n} H$ the ample $\Q$-divisor. Then there exists a model $Z_n$ of $X$ and an ample $\Q$-divisor model $D_{Z,n}$ of $D_{n}$. 

    Since $Z_n$ and $\overline{X}$ are birational, $\sigma$ induces a marked point $\sigma_{n,Z}: B\to Z_n$. Similarly $\pi$ induces a projection $\pi_{n,Z}: Z\to B$.
    Applying Lemma \ref{Bac} to $\alpha_1=\cdots = \alpha_{b-j-1}=D_{Z,n}|_{\sigma_{n,Z}(B)}$, where $j=0,\cdots,b-2$, and $\beta= \pi_{n,Z}^*(M)$, we have
    \begin{align*}
     {\left(D_{Z,n}|_{\sigma_{n,Z}(B)}\right)}^{b-j-1} \leq C_b \frac{\sigma_{n,Z}(B) \cdot {\left(D_{Z,n}|_{\sigma_{n,Z}(B)}\right)}^{b-j-1} \cdot \pi_{n,Z}^*(M)^{j+1}}{\sigma_{n,Z}(B)\cdot \pi_{n,Z}^*(M)^b} {\left(\pi_{n,Z}^*(M)|_{\sigma_{n,Z}(B)}\right)}^{b-j-1}.
    \end{align*}
    Intersecting the above inequality by 
    \begin{align*}
        D_{Z,n}|_{\sigma_{n,Z}(B)}\cdot {\left(\pi_{n,Z}^*(M)|_{\sigma_{n,Z}(B)}\right)}^{j},
    \end{align*}
    on the two sides, we obtain
    \begin{align*}
         \sigma_{n,Z}(B) \cdot {\left(D_{Z,n}\right)}^{b-j} \cdot  {\left(\pi_{n,Z}^*(M)\right)}^{j}  \leq C_{b,n,Z} \sigma_{n,Z}(B) \cdot {\left(D_{Z,n}|_{\sigma_{n,Z}(B)}\right)}^{b-j-1} \cdot \pi_{n,Z}^*(M)^{j+1},
    \end{align*}
    where
    \begin{align*}
        C_{b,n,Z}\coloneqq C_b \frac{\sigma_{n,Z}(B) \cdot D_{Z,n} \cdot \pi_{n,Z}^*(M)^{b-1}}{\sigma_{n,Z}(B)\cdot \pi_{n,Z}^*(M)^b}=C_b\frac{\sigma(B)\cdot D_{n}\cdot\pi^{*}(M)^{b-1}}{\sigma(B)\cdot\pi^{*}(M)^b}.
    \end{align*}
    Its limit is 
    \begin{align*}
        C_b'\coloneqq \lim_{n\to +\infty} C_{b,n,Z}= C_b\frac{\sigma(B)\cdot D \cdot\pi^{*}(M)^{b-1}}{\sigma(B)\cdot \pi^{*}(M)^{b}}\leq C_bC_M/\deg(M),
    \end{align*}
    Which is independent of $\sigma$.
By induction
\begin{align*}
    \sigma(B)\cdot D_n^b = \sigma_{n,Z}(B) \cdot {\left(D_{Z,n}\right)}^{b} &\leq C_{b,n,Z} \sigma_{n,Z}(B) \cdot {\left(D_{Z,n}\right)}^{b-1} \cdot \pi_{n,Z}^*(M)\\
    &\leq \cdots\\
    &\leq C_{b,n,Z}^{b-1} \sigma_{n,Z}(B) \cdot D_{Z,n}\cdot \pi_{n,Z}^*(M)^{b-1}\\
    &=C_{b,n,Z}^{b-1} \sigma(B)\cdot D_n \cdot \pi^{*}(M)^{b-1}.
\end{align*}
Passing to the limit we get
\begin{align*}
    \sigma(B)\cdot D^b \leq {C_b'}^{b-1} \sigma(B)\cdot D \cdot \pi^{*}(M)^{b-1} \leq {C_b'}^{b-1}C_M.
\end{align*}
Hence we can take $C(M,D)\coloneqq {C_b'}^{b-1}$.

(3) Let $A_i$ be ample $\R$-divisors and $E_i$ effective $\R$-divisors $1 \leq i\leq k$, such that $D=A_i+E_i$ and $\Bbb_+(D)=\cap_{i=1}^k \supp (E_i)$. Write $A_i = \sum_{j=1}^{n_i} a^i_j A^i_j,$
where $a^i_j>0$ and $A^i_j$ are ample divisors.
Similarly, write $E_i = \sum_{j=1}^{m_i} e^i_j E^i_j,$
where $e^i_j>0$ and $E^i_j$ are effective divisors.
Up to taking a large multiple of $D$, we can suppose that $a^i_j,e^i_j>1$. For any integer $x$, we denote by $\lfloor x \rfloor$ the largest integer less than $x$. Define big divisors 
\begin{align*}
    D_i\coloneqq \sum_{j=1}^{n_i} \lfloor a^i_j \rfloor A^i_j + \sum_{j=1}^{m_i} \lfloor e^i_j \rfloor E^i_j.
\end{align*}
There exists a sufficiently large and divisible integer $l$ such that for any $1\leq i\leq k$, the restriction of the rational map $\overline{X} \to \Pbb H^0(\overline{X},\Ocal_{\overline{X}}(l D_i))$ to $\overline{X}\setminus \Bbb_+(D_i) \supset \overline{X}\setminus \supp E_i$ is an isomorphism onto its image, see \cite{BoucksomMathZ14}. Thus for all those $\sigma$ such that $\sigma(B)\not\subset \supp E_i$, there exists a constant $c_i>0$ such that $\sigma(B)\cdot A^b < c_i \sigma(B)\cdot D^b$. Thus it suffices to set $C(D,A)\coloneqq \sum_i c_i.$
\end{proof}

\section{Sparsity of stable marked points}
\subsection{Parameter spaces}
\subsubsection{Chow variety}
See \cite{Andreotti,Barlet_I,Barlet_II, kollar1999rational} for more details. 

Let $X$ be a complex projective variety. 
Let $n,d \in \Z_{\geq 1}$ and fix an ample line bundle $A$ on $X$. Denote by $\Cscr_{n,d}(X,A)$ the set of cycles of dimension $n$ and of degree $d.$ In other words, an element $C\in \Cscr_{n,d}(X)$ is a formal finite sum $ C = \sum_i m_iC_i$, where $m_i\in \Z_{\geq 1}$ and $C_i$ are irreducible subvarieties of dimension $n$ such that $\sum_i m_i \deg_A C_i =d$.
We endow $\Cscr_{n,d}(X)$ with an algebraic structure so that, in particular, the subset 
\begin{align*}
    \Xscr(\Cscr_{n,d}(X,A)) \coloneqq \{(x,\sigma)\in X\times \Cscr_{n,d}(X,A) \ |\ x\in\sigma \}
\end{align*}
is a (reducible) subvariety of  $X\times \Cscr_{n,d}(X,A)$.
We call $\Cscr_{n,d}(X,A)$ the \emph{Chow variety} (of dimension $n$ and degree $d$ of $X$) and $\Xscr(\Cscr_{n,d}(X,A))$ its \emph{graph}. 
Denote by $\Cscr_n^d(X,A) \coloneqq \sqcup_{d'\leq d} \Cscr_{n,d'}(X,A)$.

If $f\colon X \to Y$ is a morphism of complex projective varieties, there is an induced morphism
\begin{align}\label{pushforward}
    f_{*}:\Cscr_n(X) \to \Cscr_n(Y)
\end{align}
defined as follows. First let $C\in \Cscr_n(X)$ be an irreducible variety of dimension $n$. If $f(C)$ has dimension strictly less than $n$, we set $f_{*}(C)$ to be the empty $n$-cycle. Otherwise, the restriction $f|_{C}:C \to f(C)$ is generically finite of degree $m$ and we define $f_{*}(C)\coloneqq mf(C)$. We then extend  the map $f_{*}$ by linearity to all cycles.

The analytic counterpart of $\Cscr_n^d(X,A)$ is usually called the space of (compact) cycles, or Barlet space (of $n-$cycles of degree at most $d$). 

\subsubsection{Hilbert scheme}
See \cite{DebarreHigher01,SchemeHilbert,Hanamura87,StrommeFunctors} for more details.
Let $X$ be a complex projective variety. For any complex scheme $T$, define
\begin{align*}
    \underline{\mathrm{Hilb}}_{X}(T)&\coloneqq \{\text{closed subschemes }Y \subset X\times T  \text{ which are flat over }T \}.\\
    \underline{\mathrm{Aut}}_{X}(T)&\coloneqq \{T\text{-automorphisms of } X\times T\} \numberthis \label{authilbert}
\end{align*}
They define contravariant functors from the category of complex schemes to the category of sets. By \cite{SchemeHilbert}, $\underline{\mathrm{Hilb}}_{X}$ is representable by a locally Noetherian complex scheme $\mathrm{Hilb}(X)$, called the \emph{Hilbert scheme} of $X$. It means that there is a family (closed subscheme) $\Xscr(\mathrm{Hilb}(X/S))\subset X\times \mathrm{Hilb}(X)$, such that $\pi_X: \Xscr(\mathrm{Hilb}(X)) \to \mathrm{Hilb(X)}$ is flat. Moreover, the family is universal in the sense that any other closed subscheme $Y\subset X\times T$, such that $\pi': Y\to T$ is flat, is the pullback by a unique morphism $\alpha: T \to \mathrm{Hilb(X)}$, i.e. we have the cartesian commutative diagram of $S$-schemes
\[\begin{tikzcd}
Y \arrow[dd, "\pi'"'] \arrow[rr] &  & \mathscr{X}(\mathrm{Hilb}(X)) \arrow[dd, "\pi_X"] \\
                                &  &                                                     \\
T \arrow[rr, "\alpha"]          &  & \mathrm{Hilb}(X)                                 
\end{tikzcd}\]

We know also that $\underline{\mathrm{Aut}}_{X}(T)$ is representable by an open subscheme $\mathrm{Aut}(X)$ of the Hilbert scheme of $\mathrm{Hilb}(X\times X)$ and it has a sutructure of group scheme (see e.g., \cite[Proposition 2.3]{Hanamura87}). In particular, we have
\begin{prop}
    Let $f: X\times \Lambda \to X\times \Lambda$ be a family of automorphisms parameterized by a quasi-projective variety $\Lambda$ (or more generally a complex scheme), then there is a uniquely determined morphism $F: \Lambda \to \mathrm{Aut}(X)$.
\end{prop}
Let $\pi: X\to B$ be a surjective morphism of complex projective varieties. Denote by $\mathrm{Rat}(B,X)$ the open subscheme of $\mathrm{Hilb}(B\times X)$ whose geometric points correspond to graphs in $B \times X$ of rational sections of $\pi$ (see the proof of \cite[Proposition 1.7]{Hanamura87}). 

Denote by ${\mathrm{Rat}(B,X)}_{\mathrm{red}}$ the reduced scheme associated to $\mathrm{Rat}(B,X)$. There is a morphism $\Theta_n(B,X): {\mathrm{Rat}(B,X)}_{\mathrm{red}} \to \Cscr_n(B\times X,A) $ defined by $\Theta_n(B,X)(C)=C.$

\subsection{Stable marked points}
Let $A$ be an ample line bundle on $\overline{X}$. Recall the constants $C_f$, $C(D,A)$, $C(M,D)$ defined in Proposition \ref{constructionHeight} and \ref{uniformbound}. Define the constant
\begin{align*}
    C_A\coloneqq 3C(D,A)C(M,D)C_f
\end{align*}

\begin{lem}\label{Zariskiclosed}
    Let $k\geq 1$ be a positive integer. Then the set
    \begin{align*}
        \Cfra_k \coloneqq \{C\in \Cscr_b^{C_A}(\overline{X},A) \ |\ \deg_A f^n(C) < C_A, \forall -k \leq n\leq k\} 
    \end{align*}
    is a Zariski closed subset of $\Cscr_b^d(X,A)$. In particular,
    \begin{align*}
        \Cfra_\infty \coloneqq \{C\in \Cscr_b^{C_A}(\overline{X},A) \ |\ \deg_A f^n(C) < C_A, \forall n\in \Z\}
    \end{align*}
    is Zariski closed.
\end{lem}
\begin{proof}
    We can not directly use the pushforward formula \eqref{pushforward} since $f$ is only a birational map of $\overline{X}$. Nevertheless, we can rely on the intermediate space $X_i$ constructed in the proof of Proposition \ref{constructionHeight} (recall the notations $X_i$, $\varphi_i$, $\pi_i$, $g^\pm_i$ therein).
    
    Denote by $\Gscr_k$ the set of functions $X_k \to \overline{X}$ of the form $h_k\circ \cdots \circ h_1$, where $h_i= g^+_i$, $g^-_i$ or $\varphi_i$. Define an ample line bundle $A_k$ on $X_k$ to be 
    $$A_k\coloneqq \sum_{g\in \Gscr_k}g^* A.$$
    By \eqref{pushforward}, an element $C$ of the Zariski closed subset
    \begin{align*}
        G_k \coloneqq \bigcap_{g\in \Gscr_k}g^{-1}(\Cscr_b^{C_A}(\overline{X},A))
    \end{align*}
    is of degree $\deg_{A_k}C < 3^k C_A$, hence $C\in \Cscr_b^{3^k C_A}(X_k,A_k)$.
    To conclude, it suffices to remark that
    \begin{align*}
        \Cfra_k = {(\varphi_1\circ \cdots \circ \varphi_k)}_*(G_k).
    \end{align*}

\end{proof}

\begin{prop}\label{gap}
    There exists a positive constant $\varepsilon_f>0$ such that given any marked point $\sigma$ with $\sigma(B)\not\subset \Bbb_+(D)$, if $\hhat_f(\sigma)<\varepsilon_f$, then $\hhat_f(\sigma)=0$.
\end{prop}

\begin{proof}
    Since $\Cfra_1\supset \Cfra_2 \supset \cdots$, there exists an integer $N\geq 1$ such that $\Cfra_\infty=\cap_{k=1}^N \Cfra_k$. By Proposition~\ref{constructionHeight}, for any $n\in \Z$, we have
    \begin{align*}
        f^n(\sigma(B))\cdot D\cdot \pi^*(M)^{b-1} < 2C_f + \lambda_+^n\hhat_f^+(\sigma) + \frac{1}{\lambda_-^n}\hhat_f^-(\sigma).
    \end{align*}
    Choose $\varepsilon_f$ sufficiently small so that for all $-N\leq n \leq N$, $\lambda_+^n\hhat_f^+(\sigma) + \frac{1}{\lambda_-^n}\hhat_f^-(\sigma)<C_f$. Then $f^n(\sigma(B))\cdot D\cdot \pi^*(M)^{b-1} < 3C_f,$ and $\sigma \in  \Cfra_\infty$. By Propositions~\ref{uniformbound} and~\ref{stablebound}, $\sigma$ is stable.
\end{proof}

\begin{rem}\normalfont\label{remarkHilbert}
    A cycle $C$ in $\Cfra_\infty$ can be uniquely expressed as $C=C_h+C^1_v+C^2_v$, where $C_h$ consists solely of horizontal components corresponding to stable marked points by Proposition~\ref{uniformbound};    $C^1_v$ contains only vertical components over $\Lambda$ that are contained in $\Bbb_+(D)$, while $C^2_v$ contains only vertical components supported above $B\setminus \Lambda.$ However, horizontal cycles coming from a stable marked point $\sigma$ whose image is contained in $\Bbb_+(D)$ may not be in $\Cfra_\infty$. We use another parameter space --- namely, Hilbert scheme, as will be discussed in Proposition~\ref{propsparsity} --- to remove the vertical parts, at the expense of losing the compactness.
\end{rem}

\begin{cor}\label{corgap}
    Given any marked point $\sigma$, $\hhat_f(\sigma)=0 \iff \hhat_f^+(\sigma)=0 \iff \hhat_f^-(\sigma)=0$.
\end{cor}
\begin{proof}
    Suppose that $\hhat_f^+(\sigma)=0$. Then $\hhat_f(f^n\circ \sigma) = \hhat_f^+(f^n\circ \sigma) + \hhat_f^-(f^n\circ \sigma)
        = \frac{1}{\lambda_-^n} \hhat_f^-(\sigma).$
    By Proposition \ref{gap}, if we take $n$ large enough so that $\frac{1}{\lambda_-^n} \hhat_f^-(\sigma)< \varepsilon_f$, then $\hhat_f(f^n\circ \sigma)=0$. Hence $\hhat_f(\sigma)=0$.
\end{proof}

Recall (Remark~\ref{remmodel}) that $\Bbb_+(D)$ has no vertical irreducible components in $X.$
\begin{lem}\label{corB_+}
    The sets $\Bbb_+(D)\cap X$ and $\Bbb_+(D_\eta)$ is invariant by $f$ and $f_\eta,$ i.e., 
    \begin{align*}
        f(\Bbb_+(D)\cap X)=f^{-1}(\Bbb_+(D)\cap X)=\Bbb_+(D)\cap X,\\
        f_\eta(\Bbb_+(D_\eta))=f_\eta^{-1}(\Bbb_+(D_\eta))=\Bbb_+(D_\eta).
    \end{align*}
\end{lem}
\begin{proof}
    The invariance of $\Bbb_+(D_\eta)$ is exactly \cite[Lemma 2.16]{LesieutreSatriano21}. By Proposition~\ref{B_+genericfiber}, we also obtain the invariance of $\Bbb_+(D)\cap X$.
\end{proof}

\begin{prop}\label{propsparsity}
    Let $\pi: X \to \Lambda$ be non-birationally isotrivial. Then, the union of images of stable marked points is not Zariski dense in $X$. Equivalently, the $\C(B)$-points --- which are generic points of the images of stable marked points --- are not Zariski dense in the generic fiber $X_\eta$.
\end{prop}
\begin{proof}
Denote by $p$ the projection $p: B\times \overline{X} \to \overline{X}.$
Denote by $\gamma_b$ the composition
\begin{align*}
    \gamma_b: \mathrm{Rat}(B,\overline{X})_{\mathrm{red}}\overset{\Theta_b(B,\overline{X})}{\longrightarrow} \Cscr_b(B\times \overline{X})\overset{p_*}{\longrightarrow}\Cscr_b(\overline{X}).
\end{align*}
Let $\Gamma_b \coloneqq \gamma_b^{-1}(\Cfra_\infty).$ Denote by $q$ the projection $q: B\times X\times \Gamma_{b,\mathrm{red}} \to X\times \Gamma_{b,\mathrm{red}}$. Then $q$ induces a proper family $$\overline{q}: q(\Xscr(\Gamma_{b,\mathrm{red}}))\subset \overline{X}\times \Gamma_{b,\mathrm{red}} \to \Gamma_{b,\mathrm{red}}.$$ 

Suppose the union of the images of stable marked points is not Zariski dense in $X$. Then, there exists an irreducible component $Z$ of $\Gamma_{b,\mathrm{red}}$ such that the projection $\overline{p}_{\overline{X}}\colon \Xscr(Z) \subset \overline{X}\times Z\to \overline{X} $ is dominant. 

Let us show that $\overline{p}_X$ is generically finite of degree one, thereby proving it is birational. This argument is inspired by \cite{DujardinLyubich,GVhenon}. Let $x_0\notin \Bbb_+(D)$ be a $f_{t_0}$-saddle periodic point in the fiber $X_{t_0}$, where $t_0\in \Lambda(\C)$. By Lemma~\ref{corB_+}, up to taking an iterate of $f$, we may assume it is a fixed point. By the implicit function theorem, there exists a small analytic open subset $U$ containing $t_0$ and a holomorphic section $\sigma_0: U \to X_U(\C)$ such that $\sigma_0(t)$ is an $f_t$-fixed saddle point for all $t\in U.$ Let $u_1(t), \cdots , u_k(t)$ be the eigenvalues of $f_t$ at $\sigma_0(t)$ that lie outside the closed unit disk, counted with multiplicities $n_1,\cdots, n_k$. Similarly,  let $s_1(t), \cdots , s_l(t)$ be the eigenvalues of $f_t$ at $\sigma_0(t)$ that lie inside the open unit disk, counted with multiplicities $m_1,\cdots, m_l$. These eigenvalues depend holomorphically on $t.$ 

Thus, by a holomorphic coordinate change, we can reduce to the following situation.
\begin{enumerate}
    \item We have a projection $\pi :(\C^{\dim X -\dim B}\times \C^{\dim B}, 0) \to (\C^{\dim B}, 0) $.
    \item We have a biholomorphic function $f_t$ of $(\C^{\dim X-\dim B})$ parameterized by $t\in (\C^{\dim B}, 0)$ such that \begin{enumerate}
        \item $f_t(0)=0$.
        \item the linear part of $f_t$ is the multiplication by a Jordan matrix $J$ such that its diagonal elements are ordered as $u_1(t),\cdots,u_k(t),s_1(t),\cdots,s_l(t)$, counted with multiplicities.
    \end{enumerate}
\end{enumerate}
Denote $N_u\coloneqq \sum_{i=1}^k n_i$. We call the set $\left\{(z_1,\cdots,z_{N_u},0,\cdots,0)\in \C^{\dim X -\dim B} \mid z_i \in (\C,0) \right\}$ the \emph{local unstable manifold} and the set $\left\{(0,\cdots,0,z_{N_u+1},\cdots,z_{\dim X -\dim B})\in \C^{\dim X -\dim B} \mid z_i \in (\C,0) \right\}$ the \emph{local stable manifold}. A simple but important observation is that the intersection of the local stable and unstable manifolds is the origin.

Let $\sigma$ be a stable marked point whose image passes through $x_0$. By proposition~\ref{uniformbound} and Lemma~\ref{corB_+},the iterates $f^n\circ\sigma$, for $n\in \Z$, are all stable, and thus belong to the set $\Cfra_\infty,$ which is compact. Consequently, there is a limit point $\sigma_\infty \in \Cfra_\infty$. We claim that $\sigma_\infty=\sigma_0$. To see this, suppose for a general $t\in U$, $\sigma(t)$ is not on the local stable manifold. Then the cycle $\sigma_\infty$ would have a vertical component passing through $x_0\not\in \Bbb_+(D)$, which contradicts Remark~\ref{remarkHilbert}. By symmetry, for a general $t\in U$, $\sigma(t)$ must also be on the local unstable manifold. Thus, $\sigma_\infty$ must be $\sigma_0$, as claimed.

We have established that there exists at most one stable marked point passing through a saddle periodic point in $X$. By the Zariski density of saddle points, we deduce that $\overline{p}_X$ is generically finite of degree one, and hence birational. Let $V\subset X$ be a Zariski open subset such that the restriction 
\begin{align*}
    \overline{p}_V\coloneqq \overline{p}_X|_{{\overline{p}_X}^{-1}(V)}:{\overline{p}_X}^{-1}(V)\subset V\times Z \to V
\end{align*}
is an isomorphism. In particular, if $t\in \pi(V)$, setting $\Xscr(Z)_t\coloneqq \Xscr(Z)\bigcap \left(V_t \times Z\right)$, then the map 
\begin{align*}
    \Xscr(Z)_t \subset V_t\times Z \overset{\overline{q}}{\longrightarrow} Z
\end{align*}
is injective. Now for any two $t_1,t_2\in \pi(V)(\C),$ the map
\begin{align*}
    V_{t_1} \overset{\overline{p}_V^{-1}}{\longrightarrow} \Xscr(Z)_{t_1} \overset{\overline{q}}{\longrightarrow} \overline{q}(\Xscr(Z)_{t_1})\cap\overline{q}(\Xscr(Z)_{t_2}) \overset{\overline{q}^{-1}}{\longrightarrow} \Xscr_{t_2} \overset{\overline{p}_V}{\longrightarrow} V_{t_2}
\end{align*}
is birational.
\end{proof}

\begin{cor}\label{corK3}
    Let ($\pi,f$) be a family of hyperbolic automorphisms of smooth complex projective surfaces. Let $\sigma$ be a marked point. If $\sigma(B)\subset \Bbb_+(D)$, then $\hhat_f(\sigma)=0$. If $\sigma(B)\not\subset \Bbb_+(D)$ and the family is non-birationally isotrivial, then $\hhat_f(\sigma)=0$ if and only if $\sigma$ is periodic. If for a general fiber, its automorphism group equals to its group of birational maps (e.g., a K3 surface), 
    we can relax the condition to require only that the family is non-isotrivial.
\end{cor}

\begin{proof}
    Let $x\in X_\eta$ be the generic point of $\sigma(B)$ in the generic fiber of $\pi.$ Suppose $x\in \Bbb_+(D_\eta)$. Let $C_\eta \subset \Bbb_+(D_\eta)$ be a curve containing $x$. Consequently We have a family of curves $\pi_C: C\subset \overline{X} \to B$ contained within $\Bbb_+(D)$, where $C$ is the Zariski closure of $C_\eta$. After possibly reducing $\Lambda$ and taking an iterate of $f$, we can assume that $C_t$ is fixed by $f_t$ and smooth, for all $t\in \Lambda$. By \cite[Lemma 5.3]{KawaguchiSurface08}, the restriction of the line bundle $L|_C$ to all fibers $C_t, t\in \Lambda(\C)$, is trivial. Therefore, there exists a vertical divisor $V_C$ on $C$ such that $L|_C = \Ocal_C(V_C)$. Now, given a marked point $\sigma$ such that $\sigma(B)\subset C$, we have $ h(\sigma)=\sigma(B) \cdot L|_C = \sigma(B) \cdot \Ocal_C(V_C)$. As in the proof of Proposition~\ref{constructionHeight}, this is bounded by a constant independent of $\sigma,$ implying $\hhat_f(\sigma)=0.$
    
    Now, suppose $x\notin \Bbb_+(D_\eta)$. By Proposition~\ref{B_+genericfiber} and Lemma~\ref{corB_+}, for any $n\in \Z$, $f_\eta^n(x)\not\in \Bbb_+(D_\eta)$. Furthermore, by Proposition~\ref{propsparsity} and the non-birational isotriviality, the orbit of $x$ under $f_\eta$ is not Zariski dense in $X_\eta.$ Hence, the Zariski closure of the orbit is $f_\eta$-invariant. By Example~\ref{exampleB_+}, it is finite since can not be a curve.
\end{proof}

\begin{cor}\label{cor:KSconj}
    Let ($\pi$,$f$) be a family of hyperbolic automorphisms of smooth complex projective varieties satisfying (H1) and (H2). Suppose it is not birationally isotrivial (or only not isotrivial in the case of a family of K3 surfaces). Let $\sigma$ be a section with dense orbit. Then $\alpha_f(\sigma)=\lambda_+$.    
\end{cor}
We require the following fundamental inequality, which is established in more general settings in \cite{zbMATH07946735,kawaguchiSilverman,MatsuzawaUpperbounds}. In our specific case, the proof follows directly from (the proof of) Proposition~\ref{uniformbound}.
\begin{lem}
    We have $\overline{\alpha}_f(\sigma) \leq \lambda_+$.
\end{lem}
\begin{proof}[Proof of Corollary~\ref{cor:KSconj}]
    By Proposition~\ref{propsparsity} and Corollary~\ref{corK3}, $\hat{h}_f(\sigma)\neq 0$. By the weak Northcott property over function fields Corollary~\ref{corgap}, $\hat{h}_f^+(\sigma)\neq 0$ either. The rest of the proof is standard and similar to the number field case, see e.g., \cite[Section 5]{matsuzawa2023recentadvanceskawaguchisilvermanconjecture}. In fact, by Proposition~\ref{constructionHeight}, we have
    \begin{align*}
        \underline{\alpha}_f(\sigma)\geq \liminf_{n\to +\infty}h(f^n\circ\sigma)^{1/n}\geq \liminf_{n\to+\infty}\left(\hat{h}^+_f(f^n\circ\sigma) \right)^{1/n} = \liminf_{n\to+\infty}\left(\lambda_+^n\hat{h}^+_f(\sigma) \right)^{1/n}=\lambda_+.
    \end{align*}
    Thus, combining with the previous lemma, the proof is complete.
\end{proof}

\section{Green currents}
\begin{prop}\label{constructionGreenCurrents}
    Let ($\pi,f$) be a family of hyperbolic automorphisms of smooth complex projective varieties. Assume hypothesis $(\mathrm{H}1)$. Then there exists a positive closed (1,1)-current $\hat{T}^\pm_f$ on $X(\C)$ with continuous local potential and such that $\frac{1}{\lambda_\pm}{(f^{\pm 1})}^*\hat{T}^\pm_f=\hat{T}^\pm_f$. It is unique in the sense that if there is another current $R^\pm$ satisfying the above properties and $R^\pm - \hat{T}^\pm_f$ is $\ddc$-exact, then $\hat{T}^\pm_f=R^\pm$. 
\end{prop}
\begin{proof}
    Let $T$ be a positive closed current in the class $c_1(L)\in H^{1,1}(\overline{X},\R)$. 
    By \eqref{VerDiv},
    \begin{align*}
         \frac{1}{\lambda_\pm}{g^\pm_1}^*(L^\pm)=\varphi_1^* L^\pm -\Ocal_{X_1}(V^\pm).
    \end{align*}
    Set $L^\pm_V\coloneq \Ocal_{X_1}(V^\pm)$. Pulling back by $\frac{1}{\lambda_\pm}{g^\pm_2}^*$, we get
    \begin{align*}
        \frac{1}{\lambda^2_\pm}{g^\pm_2}^*{g^\pm_1}^*(L^\pm) &= \frac{1}{\lambda_\pm}\varphi_2^*{g^\pm_1}^* L^\pm - \frac{1}{\lambda_\pm} {g^\pm_2}^* L^\pm_V = \varphi_2^*\varphi_1^* L^\pm - \varphi_2^*L^\pm_V- \frac{1}{\lambda_\pm} {g^\pm_2}^* L^\pm_V
    \end{align*}
    By induction, we have
    \begin{align*}
    \frac{1}{\lambda_\pm^n}{g^\pm_n}^*\cdots {g^\pm_1}^{*}L^\pm =\varphi_n^{*}\cdots \varphi_1^{*}L^\pm + \sum_{j=0}^{n-1}\frac{1}{\lambda_\pm^j}\varphi_n^*\cdots \varphi_{j+2}^*{g^\pm_{j+1}}^*\cdots {g^\pm_2}^* L^\pm_V.
    \end{align*}
\end{proof}
Let $L_A$ be an ample line bundle on $\overline{X}$ and choose a K\"ahler form $\omega_A$ in the class $c_1(L_A)$. 
By \eqref{VerBound},
\begin{align*}
    \left|{g^\pm_n}^*\cdots {g^\pm_2}^* L^\pm_V \cdot \varphi_n^*\cdots \varphi_{1}^*L_A^{\dim X-1}\right|  < N\cdot L_A^{\dim X-1}.
\end{align*}
Thus
\begin{align}\label{bdmassine}
    \frac{1}{\lambda_\pm^n}{g^\pm_n}^*\cdots {g^\pm_1}^{*}L^\pm \cdot \varphi_n^*\cdots \varphi_{1}^*L_A^{\dim X-1} \leq L^\pm\cdot L_A^{\dim X-1} + \frac{\lambda_\pm}{\lambda_\pm -1} N\cdot L_A^{\dim X-1}.
\end{align}
The mass of $\frac{1}{\lambda_\pm^n}{(f^{\pm n})}^{*}(T)$ on $X$ is 
\begin{align*}
    \int_{X}\frac{1}{\lambda_\pm^n}{(f^{\pm n})}^*T \wedge \omega_A^{\dim X-1} &= \int_{X_{n,\Lambda}} \frac{1}{\lambda_\pm^n}\varphi_n^{*}\cdots \varphi_1^{*}{(f^{\pm n})}^{*}(T)\wedge \varphi_n^{*}\cdots \varphi_1^{*}\omega_A^{\dim X-1}\\
    &\leq \int_{X_{n}} \frac{1}{\lambda_\pm^n}{g^\pm_n}^*\cdots {g^\pm_1}^{*}(T)\wedge \varphi_n^{*}\cdots \varphi_1^{*}\omega_A^{\dim X-1}\\
    &= \frac{1}{\lambda_\pm^n}{g^\pm_n}^*\cdots {g^\pm_1}^{*}L \cdot \varphi_n^{*}\cdots \varphi_1^{*}L_A^{\dim X -1},
\end{align*}
which is uniformly bounded by \eqref{bdmassine}.
Hence there exists an subsequence of $\frac{1}{n}\sum_{j=0}^{n-1}\frac{1}{\lambda^j}{(f^{\pm j})}^{*}(T)$ which converges to a positive closed current $\Tilde{T}_f^\pm$ such that $\frac{1}{\lambda_\pm}{f^{\pm}}^{*}\Tilde{T}_f^\pm = \Tilde{T}_f^\pm.$

Now, choose a smooth form $\omega^\pm$ in the class $c_1(L^\pm)$. Again by the equation \eqref{VerDiv}, there exists a smooth real function $v_0$ on $X$ such that
\begin{align*}
    \frac{1}{\lambda_\pm}{(f^{\pm 1})}^* \omega^\pm = \omega^\pm +\ddc v_0.
\end{align*}
Pulling back by $\frac{1}{\lambda_\pm^n}{(f^{\pm 1})}^*$, we have by induction that
\begin{align*}
    \frac{1}{\lambda_\pm^n}{(f^{\pm n})}^* \omega^\pm = \omega^\pm +\ddc \sum_{j=0}^{n-1} \frac{v_0\circ f^{\pm j}}{\lambda_\pm}.
\end{align*}
Taking the limit $n\to +\infty$, we have
\begin{align*}
    \hat{T}^\pm_f = \omega^\pm + \ddc v^\pm,
\end{align*}
where $\hat{T}^\pm_f \coloneqq \lim_n \frac{1}{\lambda_\pm^n}{(f^{\pm n})}^* \omega^\pm$ is a closed current and $v^\pm$ is a continuous function on $X.$ Similarly, there exist a continuous function $u^\pm$ such that
\begin{align*}
    \lim_{n\to +\infty} \frac{1}{\lambda_\mp} {(f^{\mp n})}^* \omega^\pm = \ddc u^\pm.
\end{align*}

Since $T$ and $\omega^++\omega^-$ are in the same class $c_1(L)$, there exists a quasi-plurisubharmonic function $h^\pm_0$ on $X$ such that
\begin{align*}
    T = \omega^+ +\omega^- + \ddc h^\pm_0. 
\end{align*}
Pulling back by $\frac{1}{n}\sum_{j=0}^{n-1}\frac{1}{\lambda^j}{(f^{\pm j})}^{*}$ and taking the limit $n_j \to +\infty$, we obtain
\begin{align}\label{TildeHat}
    \Tilde{T}^\pm_f = \hat{T}^\pm_f + \ddc h^\pm,
\end{align}
where $h^\pm$ is a quasi-plurisubharmonic function on $X.$ 
By restriction to a fiber $X_t, t\in \Lambda(\C)$, we have 
\begin{align*}
    \Tilde{T}_{f^\pm_t} = \hat{T}_{f^\pm_t} + dd^c h^\pm_t.
\end{align*}
Moreover, since both $\Tilde{T}_{f^\pm_t}$ and $\Tilde{T}_{f^\pm_t}$ are $1/\lambda_\pm {(f_t^\pm)}^*$-invariant, they are actually equal by~\cite[Th\'eor\`eme 2.4]{CantatK301} and~\cite[Theorem]{DSJAG10}. In other words, $h^\pm_t$ is harmonic, and thus constant. This implies that $\frac{h^\pm\circ f^{\pm n}}{\lambda_\pm^n}$ converge locally uniformly to zero. Therefore, $\Tilde{T}^\pm_f = \hat{T}^\pm_f$.

\section{Mass equals to height}
In this section we suppose $\dim B=1$.
Taking an ample divisor with support $B\setminus \Lambda$, we have an open immersion $\Lambda \to \mathbb{A}^N$ for some $N\in \N.$ Following \cite{gauthier2020geometric} we use a suitable DSH \cite{DSActa09} test function
\begin{align*}
    \Psi_r(t):=\frac{\log\max\left(|t|,e^{2r}\right)- \log\max\left(|t|,e^r\right)}{r}
\end{align*}
to estimate the mass loss. Note that        $$T_r:=dd^c\left(\log\max(|t|,e^r))\right)$$ is a positive closed current with finite mass independent of the radius $r>0.$
Let $\omega$ be any smooth form representing the class $c_1(L)$.
\begin{lem}\label{DegeEstimate}
    There exist positive constants $C_1,C_2>0$ such that for all $n\geq 1$, we can write
    \begin{align*}
        \hat{T}^{\pm}_f - \frac{1}{\lambda_\pm^n} {(f^{\pm n})}^{*}\omega = dd^c u^{\pm}_n,
    \end{align*}
    where $|u^{\pm}_n(x)|\leq \frac{1}{\lambda_\pm^n}(C_1 \log^+ |\pi(x)| + C_2)$, for all $x\in X(\C)$.
\end{lem}
\begin{proof}
Let us provide the proof for the forward Green current $\hat{T}_f^+$. Recall~\eqref{VerDiv} that we have $\varphi_1^* L^+- L^+_1 = \Ocal_{X_1}(V^+)$. Since there are only finitely many branched points $B\setminus \Lambda$, it suffices to study the convergence speed of $u_n$ around one such point $t_0$. Therefore we can assume that $\pi(V^+)=t_0=B\setminus \Lambda$.
Write
\begin{align*}
    V= \sum_{j=1}^{n^+_V}  e^+_j E^+_j - \sum_{j=1}^{n^-_V}  e^-_j E^-_j,
\end{align*}
where $e^\pm_j \in \N$ and $E^\pm_j$ are all the irreducible subvarieties of codimension 1 supported above $t_0.$

Let $p\in V^+$ and consider the germ ($V^+,p$) in a coordinate system centered at $p$. We reduce thus to the following situation: there is a holomorphic map $\pi :(\C^{\dim X},0) \to (\C,0)$, and an analytic (reducible) germ $V^+$ passing through $0$ such that $\pi(V^+)=0.$
Let $h^\pm_j$ be defining functions of $(E^\pm_j,0)$. The pullback of the divisor 0 by $\pi$ can be written as
\begin{align*}
    \pi^{-1}(0)= \sum_{j=1}^{n^+_V} a^+_j E^+_j + \sum_{j=1}^{n^-_V}  a^-_j E^-_j,
\end{align*}
where $a^\pm_j \geq 0$. Let $\omega^+$ be a smooth form representing $c_1(L^+)$. By  \eqref{VerDiv}, we have
\begin{align*}
    \frac{1}{\lambda_+} {g_1^+}^{*}\omega^+ - \varphi_1^{*}\omega^+ =[V^+]+dd^c u_0,
\end{align*}
where $u_0$ is an integrable function that is smooth on $X_{1,\Lambda}$. By Lelong-Poincar\'e equation, we have
\begin{align*}
    u_0 = -\left(\sum_{j=1}^{n_V^+}e^+_j\log|h_j^+| - \sum_{j=1}^{n_V^-}e^-_j\log|h_j^-| \right) +O(1).
\end{align*}
Hence, setting $l^\pm \coloneqq \max\{e^\pm_j \ |\ 1 \leq j\leq n^\pm_V \}$, we have
\begin{align}\label{inequalityLP}
    l^- \log|t\circ \pi| +O(1) \leq u_0\leq -l^+ \log|t\circ \pi| +O(1).
\end{align}
To conclude it suffices to pullback \eqref{inequalityLP} by $\frac{1}{\lambda_+^n}{(f^{n})}^*.$
\end{proof}
\begin{prop}\label{equalityMassHeight}
   For any marked point $\sigma$, we have $\hhat_f^\pm (\sigma) = \int_{\sigma(\Lambda)(\C)} \hat{T}_f^\pm$.
\end{prop}
\begin{proof}
writing $\sigma$ for $\sigma(\Lambda)(\C)$, we have
\begin{align}\label{errorterm}
    \left |\left \langle \left (  \hat{T}^{\pm}_f - \frac{1}{d^n} {(f^{\pm n})}^{*}\omega^\pm \right ) \wedge [\sigma], \Psi_r\circ \pi \right \rangle \right | = \left|\left \langle u^{\pm}_n \wedge [\sigma], dd^c\left( \Psi_r \circ \pi \right ) \right \rangle \right |
\end{align}
Denote by $B(0,r)$ the ball of radius $r$ and center 0 in the affine space $\A^N(\C)$. Then by Lemma \ref{DegeEstimate}, the error term \eqref{errorterm} is less than
\begin{align*}
    &\leq \frac{1}{r \lambda_\pm^n}(3rC_1 + C_2) \int_{\Lambda(\C) \cap B(0,e^{3r})} T_{2r} + T_r \leq \frac{C_3}{\lambda_\pm^n},
\end{align*}
for some constant $C_3 >0$ independent of the radius $r$. The conclusion follows by letting $r\to \infty.$ 
\end{proof}

\bibliographystyle{plain}
\bibliography{Mybio}
\end{document}